\documentclass[12pt]{article}

\usepackage[all]{xy}
\def\xym{\xymatrix}

\usepackage{amsmath, amsthm, amscd, amsfonts, amssymb}

\newtheorem{theorem}{Theorem}[section]
\newtheorem{corollary}[theorem]{Corollary}
\newtheorem{lemma}[theorem]{Lemma}
\theoremstyle{definition}

\newtheorem{remark}[theorem]{Remark}

\def\cald{{\mathcal{D}}}
\def\calx{{\mathcal{X}}}

\def\calm{{\mathcal{M}}}
\def\calp{{\mathcal{P}}}

\def\calf{{\mathcal{F}}}

\def\calx{{\mathcal{X}}}

\def\vsp{\vspace*{1,5mm}\\ }
\def\bk{\bigskip }
\def\mk{\medskip }

\def\n{\noindent }
\def\dd{\displaystyle}

\def\barr{\begin{array}}
\def\earr{\end{array}}

\def\FP{Fokker--Planck}

\def\divv{{\rm div}}
\def\rrd{{\rr^d}}

\def\rr{{\mathbb{R}}}
\def\nn{{\mathbb{N}}}
\def\1{^{-1}}
\def\9{{\infty}}
\def\lbb{{\lambda}}

\def\wt{\widetilde}
\def\ov{\overline}
\def\vf{{\varphi}}
\def\oo{{\omega}}
\def\ooo{{\Omega}}

\def\vp{{\varepsilon}}

\def\ff{\forall }
\def\({\left(}
\def\){\right)}
\def\<{\left<}
\def\>{\right>}

\title{\bf The invariance principle for~nonlinear  Fokker--Planck equations}
\author{{\bf Viorel Barbu}\footnote{Octav Mayer Institute of Mathematics of the Romanian Academy and Al.I. Cuza University, Ia\c si, Romania}\and  {\bf Michael R\"ockner}\footnote{Fakult\"at f\"ur Mathematik, Universit\"at Bielefeld, D-33501, Bielefeld, Germany}\ \footnote{Academy of Mathematics and Systems Science, CAS, Beijing, China}}
\date{}

\begin{document}
\maketitle

\begin{abstract}
\n One studies here, via the La Salle invariance principle for nonlinear semigroups in Banach spaces, the  properties  of the $\oo$-limit set $\oo(u_0)$ corresponding to the orbit $\gamma(u_0)=\{u(t,u_0);\ t\ge0\}$, where $u=u(t,u_0)$ is the solution to the nonlinear \FP\ equation
$$\barr{l}
u_t-\Delta\beta(u)+\divv(Db(u)u)=0\mbox{\ \ in }(0,\9)\times\rr^d,\\
u(0,x)=u_0(x),\ \ x\in\rr^d,\quad u_0\in L^1(\rrd),\ d\ge3.\earr$$Here, $\beta\in C^1(\rr)$ and  $\beta'(r)>0$, $\ff r\ne0$.  Moreover, $\beta$ is a sublinear  function, possibly de\-ge\-ne\-rate  in the origin, $b\in C^1(\rr)$, $b$ bounded, $b\ge b_0\in(0,\9),$ $D$ is bounded such that  $D=-\nabla\Phi$, where   $\Phi\in C(\rrd)$ is such that $\Phi\ge1,$  $\Phi(x)\to\9$ as $|x|\to\9$ and satisfies a condition of the form  $\Delta\Phi-\alpha|\nabla\Phi|^2\le0$, a.e. on $\rrd$. The main conclusion is that the equation has an equilibrium state and the set $\oo(u_0)$ is a non\-empty, compact  subset of $L^1(\rr^d)$ while, for each $t\ge0$, the operator $u_0\to u(t,u_0)$ is an iso\-metry on $\oo(u_0)$. In the non\-de\-ge\-ne\-rate case $0<\gamma_0\le\beta'\le\gamma_1$  studied in \cite{1}, it follows that $\lim\limits_{t\to\9}S(t)u_0=u_\9$ in $L^1(\rr^d)$, where $u_\9$ is the unique bounded stationary solution to the equation.
\end{abstract}

\section{Introduction}\label{s1}

Consider here the nonlinear \FP\ equation (NFPE)
\begin{equation}\label{e1.1}
\barr{ll}
u_t-\Delta\beta(u)+\divv(Db(u)u)=0\mbox{\ \ in }(0,\9)\times\rrd,\vsp
u(0,x)=u_0(x),\ \  x\in\rrd,\ d\ge3,\ u_0\in L^1(\rrd),
\earr\end{equation}under the following assumptions on $\beta:\rr\to\rr$ and $D:\rrd\to\rrd$
\begin{itemize}
\item[(i)] $\beta\in C^1(\rr),\   \beta'(r)>0,\
\ff\,r\in\rr\setminus\{0\},$ $\beta(0)=0,$   and
\begin{equation}
\label{e1.2a}
\mu_1\min\{|r|^\nu,|r|\}\le|\beta(r)|\le\mu_2|r|,\ \ff\,r\in\rr,
\end{equation}for $\mu_1,\mu_2>0$ and $\nu>\frac{d-1}d,\ d\ge3.$
\item[(ii)] $D\in L^\9(\rrd;\rrd)\cap W^{1,1}_{\rm loc}(\rrd;\rrd),$
${\rm div}\,D\in(L^1(\rrd)+L^\9(\rrd))
\cap(L^2(\rrd)+L^\9(\rrd))$
 and $D=-\nabla\Phi$, where $\Phi\in C(\rrd)\cap W^{1,1}_{\rm loc}(\rr^d)$, satisfies the conditions
\begin{equation}
\label{e1.2}
\barr{c}
\Phi(x)\ge1,\ \ff\,x\in\rrd,\ \ \lim\limits_{|x|\to\9}\Phi(x)=+\9,\vsp
\Phi^{-m}\in L^1(\rrd)\mbox{ for some }m\ge2,\earr\end{equation}
\begin{equation}
\label{e1.3}
\mu_2\Delta\Phi(x)-b_0|\nabla\Phi(x)|^2\le0,\ \mbox{a.e. } x\in\rrd\setminus\{0\}.\end{equation}
\item[(iii)] $b\in C^1(\rr)$, $b$ bounded, $b(r)\ge b_0>0$ for all $r\in[0,\9).$
\end{itemize}
We note that \eqref{e1.3} implies that
$({\rm div}\,D)^-\in L^\infty(\rr^d)$.
It should also be noted that assumption (i) does not preclude the degeneracy of the nonlinear diffusion function $\beta$ in the origin. For instance, any continuous, increasing function $\beta:\rr\to\rr$ of the~form
$$\beta(r)=\left\{\barr{ll}
\mu_1 r|r|^{d-1}&\mbox{ for }|r|\le r_0,\vsp
\mu_2h(r)&\mbox{ for }|r|>r_0,\earr\right.$$where $r_0>0,\ \mu_1,\mu_2>0,\ |h(r)|\le L|r|,\ \ff\,r\in\rr,$ $L>0,$   satisfies \eqref{e1.2a} for a suitable $\gamma_1$. As regards Hypothesis (ii), an example of such a function $\Phi$  is

\begin{equation}
\label{e1.1a}
\Phi(x)=\left\{\barr{ll}
|x|^2\log|x|+\mu&\mbox{ for }|x|\le\delta=\exp\(-\frac{d+2}{2d}\),\vsp
\vf(|x|)+\eta|x|+\mu&\mbox{ for }|x|>\delta,
\earr\right.
\end{equation}where the constants $\mu,\eta>0$ are sufficiently large, and  $\vf:[\delta,\9)\to\rr$ is given by
$$\dd\vf(r)=\delta^2\log\delta-\eta\delta-
\int^r_\delta\frac{ds}
{\frac s\delta\(\frac d{2\delta+\eta d}+\frac\delta{\gamma_1(d-2)}\(\(\frac r\delta\)^{d-2}-1\)\)},\ \
\ff r\ge\delta.$$(See \cite{1}, Appendix.)

Such an equation arises in statistical physics (see, e.g., \cite{4},~\cite{4a},~\cite{6}, \cite{13a}) and   is relevant in nonequilibrium statistical mecha\-nics where it describes the dynamics of particle densities $\rho=\rho(t,x)$  in disor\-dered media subject to anomalous diffusion (see, e.g., \cite{4}, \cite{4a}, \cite{6}, \cite{13a}). The condition $D=-\nabla\Phi$ in Hypothesis (ii) means that the force field $D=D(x)$ {\it is conservative} and this property is related to the reversibility of stationary states (see \cite{10}).
 The classical Einstein and Smoluchowski equations $u_t-\Delta u+{\rm div}(Du)=0$ are associated with the Boltzmann-Gibbs distributions, while NFPE \eqref{e1.1} corres\-ponds to a generalized entropy (for instance, the Tsallis entropy in the case where $\beta(r)\equiv ar^q$). This equation can be derived also from the standard master equation associated with the particle transport model \cite{4a}. It should be emphasized that in all physical models governed by NFPE \eqref{e1.1}, $u(t,x)$ is either the density of particles at time $t$ or a probability density associated with the corresponding McKean-Vlasov equation
\begin{equation}
\label{e1.4a}
\barr{l}
dX(t)=b(u(t,X(t)))D(X(t))dt+\dd\frac1{\sqrt{2}}\(\frac{\beta(u(t,X(t)))}{u(t,X(t))}\)^{\frac12}dW(t),\\
X(0)=X_0.\earr\end{equation}
More exactly, if $u:[0,\9)\to L^1(\rr^d)$ is a distributional solution to \eqref{e1.1}, which is weakly $t$-continuous, then there is a weak solution $X$ to \eqref{e1.4a} on some probability space $(\ooo,\calf,\mathbb{P},\calf_t)$ such that $u(t,x)dx=\mathbb{P}\circ(X(t))^{-1}(dx)$, $u_0(x)dx=\mathbb{P}\circ(X_0)^{-1}(dx).$ (See \cite{2}, \cite{2a}.) In this sense, all our results in this paper have a probabilistic interpretation  and, in particular, we thus prove (see Theorem \ref{t2.2a}) the existence of an invariant measure for \eqref{e1.4a} for the class of degenerate cases, where $\beta$ is as in Hypothesis (i).

An efficient functional way to treat NFPE \eqref{e1.1} in $L^1(\rr^d)$, which is the natural state space for the well-posedness of this equation, is to represent it as an infinite dimensional Cauchy problem in $L^1(\rr^d)$. 

In fact, it was shown in \cite{2} (see, Lemmas 3.1, 3.2 therein)  that the operator $A_0:D(A_0)\subset L^1\to L^1$ defined by\newpage 
\begin{equation}
\label{e1.6a}\barr{rcl}
A_0u&=&-\Delta\beta(u)+\divv(Db(u)u),\ \ff\,u\in D(A_0),\vsp
D(A_0)&=&\{u\in L^1;\ -\Delta\beta(u)+\divv(Db(u)u)\in L^1\},\earr\end{equation}satisfies
\begin{eqnarray}
&R(I+\lbb A_0)=L^1,\ \ff\lbb>0,\label{e1.6}\end{eqnarray}and, for each $\lbb>0$, there is $J_\lbb:L^1\to D(A_0)$ such that
\begin{eqnarray}
&J_\lbb(u)\in(I+\lbb A_0)\1\{u\},\ \ff u\in L^1(\rr^d),\label{e1.10a}\\	
&|J_\lbb(u)-J_\lbb(v)|_1 \le|u-v|_1,\ \ff\,u,v\in L^1,\ \lbb>0,\label{e1.7}\end{eqnarray}
where $L^1=L^1(\rrd)$ with its norm denoted by $|\cdot|_1$. We also set  $L^1_{\rm loc}=L^1_{\rm loc}(\mathbb{R}^d)$ and
$$\calp=\left\{u\in L^1;\ \int_\rrd u(x)dx=1,\ u\ge0,\mbox{ a.e. on }\rrd\right\}.$$
Then the  operator $A:D(A)\subset L^1\to L^1,$ defined by
\begin{equation}
\label{e1.9}
\barr{rcl}
Au&=&A_0u,\ \ff\,u\in D(A),\vsp
 {D(A)}&=& {J_\lbb(L^1),}\earr
\end{equation}is $m$-accretive in $L^1$ (that is, satisfies \eqref{e1.6}--\eqref{e1.7}), $(I+\lbb A)\1=J_\lbb,$ and one has also
\begin{eqnarray}
&\label{e1.9a}
(I+\lbb A)\10=0,\ (I+\lbb A)\1\calp\subset\calp,\ \ \ff\lbb>0,\\
&\label{e1.12a}
(I+\lbb A)\1(L^1\cap L^\9)\subset L^1\cap L^\9,\ \ \ff\,\lbb\in(0,\lbb_0),
\end{eqnarray}for some $\lbb_0>0.$

We denote by $C=\ov{D(A)}$ the closure of $D(A)$ in $L^1 $. We note that, by \cite[Theorem 2.2]{2}, if  $\beta\in C^2(\rr)$,
 then $C=L^1$. Then (see, e.g, \cite{1a}, p.~139), by the Crandall \& Liggett generation theorem for each $u_0\in C$, the Cauchy problem
\begin{equation}
\label{e1.10}
\barr{l}
\dd\frac{du}{dt}+Au=0,\ \ t\ge0,\vsp
u(0)=u_0,\earr\end{equation}has a unique mild solution $u\in C([0,\9);L^1)$ and $S(t)u_0=u(t),\ t\ge0,$ is a semigroup of nonlinear contractions in $L^1 $. In addition (see \cite[Theorem~2.2]{2}), it leaves  $\calp\cap C$ invariant for all $t\ge0$, that is, 

\begin{equation}
\label{e1.11}
|S(t)u_0-S(t)v_0|_1\le|u_0-v_0|_1,\ \ \ff u_0,v_0\in C,\ t\ge0,\end{equation}
\begin{equation}
\label{e1.12}
S(t)C\subset C,\ S(t)(\calp\cap C)\subset\calp\cap C,\ \ \ff\,t\ge0,\end{equation}
\begin{equation}
\label{e1.15b}
S(t)(C\cap L^\9)\subset L^\9\cap L^1,\ \ \ff\,t>0,\end{equation}
\begin{equation}
\label{e1.18b}
S(t+s)=S(t)S(s),\ \ \ff\,t,s\ge0.  \end{equation}$S(t)$ is called the contraction semigroup in $L^1$ generated by $A$ on $C$.

We have
\begin{equation}
\label{e1.14a}
S(t)u_0=\lim_{n\to\9}\(I+\frac tn\ A\)^{-n}u_0\mbox{ in }L^1,\ \ff u_0\in C,
\end{equation}uniformly in $t$ on compact intervals.
This means that, for each $h>0$ and $0<T<\infty$, $$S(t)u_0=\lim\limits_{h\to0}u_h(t)\mbox{ in  $L^1$ uniformly on $[0,T]$,}$$where
\begin{equation}
\label{e1.18a}
\barr{l}
u_h(t)=u^{i+1}_h,\ \ t\in[ih,(i+1)h),\ i=0,1,...,N-1,\ N=\left[\frac Th\right],\vsp
u^{i+1}_h+hAu^{i+1}_h=u^i_h,\ i=0,1,...,N-1.\earr
\end{equation}
We shall call $S:[0,\9)\to {\rm Lip}_1(C)$ (= the space of all Lipschitz mappings on $C$ with Lipschitz constant less than $1$) the nonlinear semigroup (semiflow) cor\-res\-pon\-ding to NFPE \eqref{e1.1}, ({\it the nonlinear \FP\ semiflow}) and $u(t)=S(t)u_0$,  $t\ge0$,   {\it the generalized} (or {\it mild}) {\it  solution to NFPE \eqref{e1.1}}.
We note that this ge\-ne\-ra\-lized solution $u$ is also a Schwartz distributional solution to \eqref{e1.1} on $[0,\9)\times\rrd$, but the uniqueness in the class of Schwartz distributional solutions requires some more regularity on $\beta$ (see \cite{3a}).

 As a matter of fact, since the solution $J_\lbb(f)$ of \eqref{e1.10a} in general might not be unique, it should be emphasized that equation \eqref{e1.10} with $A$ of the form \eqref{e1.9} {\it is only one realization of a solution to the \FP\ equation \eqref{e1.1} as a Cauchy problem in the space $L^1 $, which depends on  the choice of the resolvent  $\{J_\lbb\}_{\lbb>0}$.} More precisely, $J_\lbb$ for $\lbb>0$ is constructed as a limit of an approximation $J^\vp_\lbb$ as $\vp\to0$, and in general it depends on the choice of this approximation. In Section \ref{s2a} below we shall explain which approximation we choose in this paper.

Now, assume that $u_0$ is a probability density. Equation  \eqref{e1.1} describes the evolution of an open system {\it far from equilibrium} and the transition of a state $u(t)$ to an equilibrium state, that is, the convergence of the solution $u(t)$ for $t\to\9$ to a stationary probability density $u_\9$ is our objective here. To analyze this problem with the orbit $\gamma(u_0)=\{S(t)u_0,\ t\ge0\}$ of $S(t)u_0$, where $u_0\in C$,  we associate the $\oo$-limit~set
\begin{equation}
\label{e1.13}
\barr{lcl}
\oo(u_0)&=&\dd
\left\{u_\9=\lim S(t_n)u_0
\mbox{ in $L^1$ for some $\{t_n\}\to\infty$}\right\}\vsp
&=&\dd\bigcap_{s\ge0}\ \  \ov{\bigcup_{t\ge s}S(t)u_0}.\earr
\end{equation}
The properties and structure of the set  $\oo(u_0)\subset C$ are important for the long time dynamics of the semiflow $S(t)$, $t\ge0$,  and play a central role in the theory~of general dynamical systems in finite or infinite dimension (see, e.g., \cite{2aa}, \cite{3}, \cite{5}, \cite{11}). In particular, if $\oo(u_0)\ne\emptyset$ and consists of one element $u_\9$ only, this means that
 $$\lim_{t\to\9}S(t)u_0=u_\9\mbox{ in }L^1.$$
  In \cite{1}, it was proved that under additional assumptions on $\beta$ and, more exactly, if $\beta$  is not degenerate in the origin, that is,
\begin{equation}
\label{e1.14}
0<\gamma_0\le\beta'(r)\le\gamma_1,\ \ \ff\,r\in\rr,\end{equation}(which again  implies that $C=L^1$), then, for each $u_0\in \calp$, such that
\begin{equation}
\label{e1.15a}
  u_0 \ln(u_0)\in L^1(\rr^d),\ \|u_0\|=\int_{\rrd}u_0(x)\Phi(x)dx<\9,\end{equation}
one has $\oo(u_0)=\{u_\9\}$, where $u_\9$ is an equilibrium solution to \eqref{e1.1}, and is the unique solution in $(L^1\cap L^\9)(\rrd)$ to the stationary equation
\begin{equation}
\label{e1.15}
-\Delta\beta(u)+\divv(Db(u)u)=0\mbox{\ \ in }\cald'(\rrd).\end{equation}
Moreover, $u_\9$ is an equilibrium solution, that is, it minimizes the free energy of the system. (See Remark \ref{r1} below.)

This result, which links the initial nonequilibrium state $u_0$ with the final equilibrium state $u_\9$, can be viewed as an $H$-theorem type for
NFPE \eqref{e1.1}. It also has some deep implications  for  {\it the McKean--Vlasov stochastic differential equation} \eqref{e1.4a} associated with NFPE \eqref{e1.1}.

The situation is different in the   degenerate case considered here, that is, where condition \eqref{e1.14} is weakened to (i). As seen below  (see Theorem~\ref{t1}), in this case $\oo(u_0)$ {\it is a nonempty},
compact subset of  $L^1 $ and, for every fix point of $S(t)$, $t>0$,
it is contained in some sphere centered at this fix point.
This means that there is a compact set $\oo(u_0)$ of probability densities which attracts the trajectory which starts from a nonequilibrium state $u_0$.   This behaviour is specific to open systems far   from thermo\-dynamic equilibrium  (\cite{4}, \cite{4a}).  It should be mentioned that Hypothesis (ii) part \eqref{e1.2} excludes the special case of the porous media equation
$$u_t-\Delta\beta(u)=0,\ \  t\ge0,\ x\in\rrd.$$

\n{\bf Notations.} We denote by $L^p$, $1\le p\le\infty$, the space of  Lebesgue $p$-integrable functions      on $\rr^d$   and by $L^p_{\rm loc}$ the space $L^p_{\rm loc}(\rrd)$. The norm in $L^p$ is denoted by $|\cdot|_p$ and the scalar product in $L^2$ is denoted by  $\left<\cdot,\cdot\right>_2$.
Let $C^k(\rr^\ell)$, $k=1,2,$ $\ell\ge1$, denote the space of $k$-differentiable functions on $\rrd$ and   $C_b(\rrd)$
the space of continuous and bounded functions on $\rr^d$.
Let $W^{k,p}(\rrd)$, $k=1,$ $1\le p\le\9$,
denote the classical Sobolev spaces on $\rrd$ and
 $\Delta,\nabla,\divv$ the standard differential operators on $\rrd$ taken in the sense of   Schwartz distributions, i.e. on   $\cald'(\rrd)$.
We set $H^k=W^{k,2}(\rr^d)$ and denote by $H^{-k}$ the dual space of $H^k$. The norm of $\rrd$ will   be denoted by $|\cdot|$. We shall also use the notation
\begin{equation}\label{e1.16}
\|u\|=\int_{\rrd}\Phi(x)|u(x)|dx,\ \ \ff\,u\in L^1,
\end{equation}
and denote by $\calm$ the subspace of $L^1$ with the norm \eqref{e1.16} finite. For each $\eta>0$, we set
\begin{equation}\label{e1.17}
\calm_\eta=\{u\in \calm;\, \|u\|\le\eta\}.
\end{equation}
 We also set
\begin{equation}
\label{e1.29}\calm_+=\{u\in\calm;\ u\ge0,\mbox{ a.e. in }\rr^d\}.
\end{equation}
Furthermore, $\calp$ denotes the set of all probability densities on $\rrd$.

\section{Construction of a solution semigroup\\ to~\eqref{e1.1}~with stationary point}\label{s2a}
	\setcounter{equation}{0}
	
	Though, as explained in the introduction, the existence of a solution semigroup to \eqref{e1.1} follows from \cite[Theorem 2.2]{2}, in this section we shall present a construction of a solution semigroup $S(t),$ $t\ge0$, for which we can prove that it has a stationary point, i.e., there exists $a\in L^1$ such that $S(t)a=a$,  $\ff t\ge0$.  So, let us fix $M\in[1,\9]$, $\vp\in(0,1]$ and assume that Hypotheses (i), (ii), (iii) hold. Consider the following approximating operator on $L^1$
	
	\begin{eqnarray}
	(A_0)_{\vp,M}u&\!\!\!=\!\!&-\Delta\beta_{\vp,M}(u)+{\rm div}(Db(u)u),\label{e21}\\
	D((A_0)_{\vp,M})&\!\!\!=\!\!&\!\!\{u\in L^1;\ -\Delta\beta_{\vp,M}(u){+}{\rm div}(Db(u)u)\in L^1\},\qquad\label{e22}
	\end{eqnarray}
	where
	\begin{equation}
		\label{e23}
		\beta_{\vp,M}(r):=\left\{\barr{ll}
		\beta(r)+\vp r,&\mbox{ if }|r|\le M,\vsp 
		\beta(M)+\beta'(M)(r-M)+\vp r,&\mbox{ if }r>M,\vsp 
		\beta(-M)+\beta'(M)(r+M)+\vp r,&\mbox{ if }r<-M,\earr\right.
	\end{equation}
	so that
	\begin{equation}
		\label{e23'}
		\beta_{\vp}:=\beta_{\vp,\9}=\beta+\vp I.
	\end{equation}
	Then, by \cite[Lemmas 3.1 and 3.2]{2} for every $\lbb>0$ there is a right inverse $J^{\vp,M}_\lbb:L^1\to D((A_0)_{\vp,M})$ of the operator $I+\lbb(A_0)_{\vp,M}:D((A_0)_{\vp,M})\to L^1$, i.e.
	$$R(I+\lbb(A_0)_{\vp,M})=L^1\mbox{\ \ and\ \ }(I+\lbb(A_0)_{\vp,M})J^{\vp,M}_\lbb= I
	$$such that $A_{\vp,M}:=(A_0)_{\vp,M\uparrow J^{\vp,M}_\lbb(L^1)}$ is $m$-accretive on $L^1$ and $J^{\vp,M}_\lbb(L^1)$ is independent of $\lbb>0$ and $J_\lbb^{\vp,M}(L^1\cap L^\9)\subset L^1\cap L^\9$. Applying the Crandall \& Liggett generation theorem (see above) to the operator $A_{\vp,M}$, we obtain the corresponding mild solution given by a nonlinear semigroup $S_{\vp,M}(t)$, $t\ge0$, on $L^1$ (see \cite[Theorem~2.2]{2} and note that $D(A_{\vp,M}):=J^{\vp,M}_\lbb(L^1)$ is dense in $L^1$, since $\vp>0$).
	
	Furthermore, by Theorem 2.1 in \cite{3a}, if $u_0\in L^1\cap L^\9$, then $S_{\vp,M}(t)u_0,$ $t\in[0,T]$, is the unique narrowly continuous (in $t\ge0$) weak solution in $(L^1\cap L^\9)((0,T)\times\rrd)$ of \eqref{e1.1}$_{\vp,M}$, where \eqref{e1.1}$_{\vp,M}$ denotes the NFPE \eqref{e1.1} with $A_{\vp,M}$ replacing $A$. In addition, for $M<\9$, by Theorem 6.1 in \cite{1} there exists $a_{\vp,M}\in\calp\cap\calm\cap L^\9(\rrd)$ which is given by
	\begin{equation}
		\label{e24}
	a_{\vp,M}(x)=g^{-1}_{\vp,M}(\mu_{\vp,M}-\Phi(x)),\ x\in\rrd,
\end{equation}
where
\begin{equation}
	\label{e25}
g_{\vp,M}(r):=\int^r_1\frac{\beta'_{\vp,M}(s)}{sb(s)}\ ds,\ \ r>0,
\end{equation} and $\mu_{\vp,M}\in\rr$ is the unique number such that
\begin{equation}
	\label{e26}
\int_{\rrd}g^{-1}_{\vp,M}(\mu_{\vp,M}-\Phi(x))dx=1,
\end{equation} 
\n such that $S_{\vp,M}(t)a_{\vp,M}=a_{\vp,M}$ for all $t\ge0$, $\lim\limits_{t\to\9}S_{\vp,M}(t)\wt u_0=a_{\vp,M}$ for all $\wt u_0\in\calp\cap\calm$ with $\wt u_0\ln\wt u_0\in L^1$ and (see \cite[Corollary 6.3]{1}) 
	\begin{equation}
		\label{e27}
|a_{\vp,M}|_\9\le\max\left(1,e^{\frac{|b|_\9}\vp\ (\mu_{\vp,M}-1)}\right).	\end{equation}
Furthermore, by \cite[Theorem 6.4]{1}, 
\begin{equation}
\label{2.8'}
-\Delta\beta_\vp(a_{\vp,M}+{\rm div}(Db(a_{\vp,M})a_{\vp,M})=0\mbox{\ \ in }\cald'(\rr^d).
\end{equation}
We note that, obviously, $g_{\vp,M}:[1,\9)\to[0,\9)$ and $g_{\vp,M}:(0,1)\to(-\9,0)$. Therefore, by \eqref{e26},
$$1\ge\int_{\{\Phi\le\mu_{\vp,M}\}}1\,dx,$$hence
\begin{equation}
	\label{e28}
\sup\{\mu_{\vp,M}\mid M\in [1,\9),\ \vp\in(0,1]\}<\9.
\end{equation}
Defining
$$\mu_{\vp}:=\sup_{M\in[1,\9)}\mu_{\vp,M}$$
we deduce from \eqref{e27} that, for all $M\in[1,\9)$,
\begin{equation}
	\label{e29}
|a_{\vp,M}|_\9\le\max\left(1,e^{\frac{|b|_\9}\vp\ (\mu_\vp-1)}\right)=:M_0.
\end{equation}
It follows by the weak uniqueness result from \cite{3a} mentioned above that
\begin{equation}
	\label{e210}
	a_\vp:=a_{\vp,M_0}=a_{\vp,M},\ \ \ff M\in[M_0,\9).
\end{equation}
Furthermore, for $M=\9,$ we obtain that for $S_\vp(t):=S_{\vp,\9}(t)$, $t\ge0$, we have
\begin{equation}
	\label{e211}
	S_\vp(t)a_\vp=a_\vp\mbox{\ \ for all }t\ge0.
\end{equation}

\begin{lemma}\label{l22} Let $\vp\in(0,1]$ and  $(A_0)_\vp:=(A_0)_{\vp,\9},$ $J^\vp_\lbb:=J^{\vp,\9}_\lbb$ and $$A_\vp:=(A_0)_{\vp\uparrow J^\vp_\lbb(L^1)},\ \ D(A_\vp):=J^\vp_\lbb(L^1),$$which $($as seen above$)$	  is $m$-accretive on $L^1$. Then:
	\begin{itemize}
		\item[\rm(i)] Let $f\in L^1\cap L^\9$. Then there exists $\lbb_0\in(0,\9)$ such that, for every $\lbb\in(0,\lbb_0]$, $J^\vp_\lbb(f)$ is the unique solution $u_\lbb\in L^1\cap L^\9$ of the equation
		\begin{equation}\label{e2.14}
		u_\lbb+\lbb(A_0)_\vp u_\lbb=f\mbox{\ \ in }\cald'(\rr^d).\end{equation}
		\item[\rm(ii)] $J^\vp_\lbb a_{\vp}=a_\vp,\ \ff\lbb\in(0,\lbb_0]$ and $a_\vp$ as above.
		\item[\rm(iii)] For every $\lbb\in(0,\lbb_0]$,
 $\dd\lim_{\vp\to0}J^\vp_\lbb f=J_\lbb f\mbox{\ \ in }L^1,$ where $J_\lbb$ is  as in \eqref{e1.6}, \eqref{e1.10a} in the introduction.
 \item[\rm(iv)]   Let $K\subset \rr^d$, $K$ compact, $q\in\(1,\frac d{d-1}\)$ and $\lbb\in(0,\lbb_0]$. Then there exists $C\in(0,\9)$ only depending on $K,q,\lbb,|D|_\9$ and $|b|_\9$ such that, for all $f\in L^1$,
 \begin{equation}\label{e2.15}
 \|\beta(J_\lbb f)\|_{L^q(K)}\le C|f|_1,
\end{equation}and,  for all $\nu\in\rr^d$, 
 \begin{equation}\label{e2.16}
\|\beta(J_\lbb f)^\nu\|_{L^q(K)}\le C\(\|E^\nu\|_{M^{\frac d{d-2}}}+\|\nabla E^\nu\|_{M^{\frac d{d-1}}}\)\xym{|f|_1\ar[r]_{\ \nu\to0}&0},  
\end{equation}where for a function $v:\rr^d\to\rr$ we set $v^\nu(x):=v(x+\nu)-v(x),$ $x\in\rr^d,$ $E(x)=\oo_d|x|^{2-d},$ $x\in\rr^d$, with $\oo_d$ $=$ the volume of the unit ball and $\|\cdot\|_{M^p},$ $p>1,$ the norm of the Marcinkievicz space $($see {\rm\cite{2}} and the references therein$)$.
	\end{itemize}\end{lemma}

\n{\it Proof.} See the Appendix.\hfill$\Box$

\begin{theorem}\label{t22} Suppose $a:=\lim\limits_{\vp\to\9} a_\vp$ exists in $L^1$. Then $J_\lbb(a)=a$, for all $\lbb\in(0,\lbb_0]$ with $\lbb_0$ as in Lemma {\rm\ref{l22} (i)}. In particular, $a\in D(A)$ and $S(t)a=a$, for all $t\ge0$.\end{theorem}

\n{\it Proof.} We have by Lemma \ref{l22} (ii), for $\lbb\in(0,\lbb_0]$, 
$$|J_\lbb a-a|_{1}\le|J_\lbb a-J^\vp_\lbb a|_{1}+|J^\vp_\lbb a-J^\vp_\lbb a_\vp|_{1}+|a_\vp-a|_{1}.$$Now, the assertion follows by Lemma \ref{l22} (iii), since $J^\vp_\lbb$ is a Lipschitz contraction for all $\vp>0$, $\lbb>0$. The last part of the assertion now follows by \eqref{e1.9} and \eqref{e1.14a}.\hfill$\Box$\bk

In Theorem \ref{t2.2a} below, we shall show that under a mild condition in addition to Hypotheses (i)-(iii) we indeed have that $a:=\lim\limits_{\vp\to0} a_\vp$ exists in $L^1$.

\section{The main results}\label{s2}
	\setcounter{equation}{0}
	
	Theorem \ref{t1} below is the main result of this work which will be proved in Section \ref{s3}.
	
\begin{theorem}\label{t1} Assume that Hypotheses {\rm(i)-(iii)} hold and let 	$\eta>0$ be arbitrary but fixed.
Let $u_0\in\calm_\eta\cap\calp\cap C$. Then, $\oo(u_0)\subset\calm_\eta\cap\calp\cap C$ is nonempty and
for all $t\ge0,$ $\oo(u_0)$ is compact in $L^1$, $\oo(u_0)=\ov{\{S(t)u_0\mid t\ge0\}}^{L^1},$ 
 and invariant under $S(t)$.
Moreover, $S(t)$  is, for every $t\ge0$,  an isometry on $\oo(u_0)$ and it is a homeomorphism from $\oo(u_0)$ onto itself for each $t\ge0$.  If~$a\in\calm_\eta\cap\calp\cap C$ is such~that 	
\begin{equation}
\label{e2.2}S(t)a=a,\ \ \ff\,t\ge0,
\end{equation} then $\oo(u_0)\subset  \{y\in\calm_\eta\cap\calp\cap C;\ |y-a|_1=r\}$, for some $0\le r\le|u_0-a|_1.$
\end{theorem}
In particular, it follows by Theorem \ref{t1} that $S(t),\ t\ge0,$ is {\it a continuous group} on $\oo(u_0)$.
Moreover, {\it the function $t\to S(t)v$ is equi-almost periodic} in $L^1$ for each $v\in\oo(u_0)$,  i.e., for every $\vp>0$ there exists $\ell_\vp>0$ such that for every interval $I$ in $\rr$ of length $\ell_\vp$ there exists $\tau\in I$ such that
$$|S(t+\tau)y-S(t)y|_1\le\vp,\ \ff t\in\rr,\ y\in\oo(u_0).$$Furthermore, 
$${\rm dist}_H(\oo(u_0),\oo(\bar u_0))\le|u_0-\bar u_0|_1,\ \ff\,u_0,\bar u_0\in\calm_\eta\cap\calp\cap C,$$where ${\rm dist}_H$ is the Hausdorff distance. 

Since the main interest is to get solutions $u(t)=S(t)u_0$ to \eqref{e1.1} in the class  $\calp$,  the initial data $u_0$ was taken in the same set $\calp$ which, by virtue of \eqref{e1.12}, implies that $u(t)\in\calp$, $\ff\,t\ge0.$ As seen below, the set $\calm_\eta$  is still invariant under the semigroup $S(t)$, but we note that this choice for initial data  (that is, $u_0\in\calm_\eta)$ was taken for technical reasons which will become clear in the proof of Theorem \ref{t1}.

In order to apply Theorem \ref{t1}, a nontrivial problem is the existence of a fixed point $a$ for the semigroup $S(t)$.

This problem is quite delicate and will be treated in Theorem \ref{t2.2a} below.
To this purpose, consider  the function $g:(0,\9)\to\rr$  defined~by
\begin{equation}
\label{e2.4a}g(r)=\int^r_1\frac{\beta'(s)}{sb(s)}\ ds,\ \ \ff\,r>0.\end{equation}We have

\begin{theorem}\label{t2.2a} Assume that, besides Hypotheses {\rm(i), (ii), (iii)},  the following conditions hold:\newpage
	
\begin{eqnarray}
\lim_{r\to+\9}g(r)=+\9,&if&\nu\in\mbox{$\left(1-\frac 1d,1\right]$};\label{e2.5}\\
\lim_{r\to0}g(r)=-\9,&if&\nu\in(1,\9).\label{e2.6}
\end{eqnarray}
\end{theorem}
Let $a_\vp\in\calp\cap\calm\cap L^\9$ be as in the previous section (see \eqref{e210}, \eqref{e211}). Then
$a:=\lim\limits_{\vp\to0}a_\vp$ exists in $L^1$ and
$$a(x)=g\1(\mu-\Phi(x)),\ \ x\in\rrd,$$where $\mu\in\rr$ is the unique number such that
$$\int_{\rr^d}a\1(\mu-\Phi(x))dx=1.$$Furthermore,
\begin{equation}\label{e3.5z}
S(t)a=a,\ \ \ff t\ge0.\end{equation}

\n{\it Proof of Theorem \ref{t2.2a}.} The last part of the assertion follows by Theorem \ref{t22}. So, it remains to prove its first part. To this end, we first note that by \eqref{e2.5}, \eqref{e2.6}, Lemma A.2 implies that
$g:(0,\9)\to\rr$ is bijective, since $g$ is strictly increasing.

Furthermore, $g((0,1))\subset(-\9,0),$ $g([1,\9))\subset[0,\9)$, $g\in C^1((0,\9))$ and for its inverse
$$g\1:\rr\to(0,\9),$$we have $g\1\in C^1(\rr),$ $(g\1)'>0,$ $g\1([0,\9))\subset [1,\9)$, $g\1((-\9,0))\subset(0,1).$ Define (cf. \eqref{e23'} and \eqref{e25})
\begin{equation}\label{e3.6z}
g_\vp(r):=\int^r_1\frac{\beta'_\vp}{sb(s)}\ ds=\int^r_1\frac{\beta'(s)+\vp}{sb(s)}\ ds,\ \ r>0.
\end{equation}
Then $g_\vp$ and its inverse $g^{-1}_\vp$ have the same properties as $g,g\1$ above. Clearly,	
\begin{equation}
	\label{e3.7z}
\xym{g_\vp\ar[r]_{\vp\to0}&g}\mbox{ locally uniformly on }\rr,
\end{equation}
hence
\begin{equation}
	\label{e3.8z}
		\xym{g^{-1}_\vp\ar[r]_{\vp\to0}&g\1}\mbox{ locally uniformly on }\rr.
\end{equation}
We recall that (see \eqref{e24}-\eqref{e26} and \eqref{e210})
\begin{equation}
	\label{e3.9z}
	a_\vp(x)=g^{-1}_\vp(\mu_\vp-\Phi(x)),\ \ x\in\rr^d,
\end{equation} 
where $\mu_\vp\in\rr$ is the unique number such that 
\begin{equation}
	\label{e3.10z}
\int_\rrd g^{-1}_\vp(\mu_\vp-\Phi(x))dx=1.\end{equation}
Since $g^{-1}_\vp([0,\9))\subset[1,\9)$, by \eqref{e3.10z} we have
$$1\ge\int_{\{\Phi\le\mu_\vp\}}1\ dx,$$
so
$$\sup_{\vp\in(0,1]}\mu_\vp<\9.$$
Suppose there exist $\vp_n\in[0,1]$, $n\in\nn$, such that $\mu_{\vp_n}\to-\9$ as $n\to\9$. Then, by \eqref{e3.10z},
$$1=\lim_{n\to\9}\int_{\rr^d} g^{-1}_{\vp_n}(\mu_{\vp_n}-\Phi(x))dx.$$
But, by Lemma A.2 applied to $\beta_\vp$ instead of $\beta$, it follows by (A.4), (A.8)  and \eqref{e1.2} that the limit on the r.h.s. is equal to zero. This contradiction implies that
$$\inf_{\vp\in(0,1]}\mu_\vp>-\9,$$so  $\{\mu_\vp\mid\vp\in(0,1]\}$ is bounded, which implies that there exist $\vp_n\in(0,1]$, $n\in\nn$, such that $\lim\limits_{n\to\9}\vp_n=0$ and $\mu:=\lim\limits_{n\to\9}\mu_{\vp_n}$ exists in $\rr$.
 
 Furthermore, by (A.5) and (A.9) (again applied to $\beta_\vp$) we have
\begin{equation}
	\label{e3.11z}
	\sup_n\int_{\rr^d} g^{-1}_{\vp_n}(\mu_{\vp_n}-\Phi(x))\Phi(x)dx<\9,
\end{equation}and by \eqref{e3.8z} we have
\begin{equation}
	\label{e3.12z}
 \xym{g^{-1}_{\vp_n}(\mu_{\vp_n}-\Phi)\ar[r]_{n\to\9}& g\1(\mu-\Phi)}
\end{equation}uniformly on compact subsets on $\rrd$. Hence the first part of the assertion follows by Lemma A.1 in the Appendix.\hfill$\Box$

\begin{corollary}\label{c1} Assume that Hypotheses {\rm (i), (ii), (iii)} 	and   \eqref{e2.5}--\eqref{e2.6} are satisfied.    Then all the conclusions of Theorem {\rm\ref{t1}} hold.  In particular, for each
	\mbox{$u_0\in\calm_\eta\cap\calp\cap C$,} the $\oo$-limit set $\oo(u_0)$ lies in the set
$$\{y\in\calm_\eta \cap\calp\cap C;\ |y-a|_1=r\},$$
where $a\in\calm_\eta\cap\calp\cap C$ satisfies \eqref{e3.5z}   and $0\le r\le|u_0-a_{\mu^*}|_1.$
\end{corollary}

\n Of course, if the set $\oo(y_0)\cap\{a\in C;\ S(t)a=a,\,\ff t>0\}$ is nonempty, it follows by Corollary \ref{c1}  that $\oo(y_0)$ contains only one element (the equilibrium state~$a$) and so $\lim\limits_{t\to\9}S(t)y_0=a$ strongly in $L^1$.

 It should be said that Theorem \ref{t1} and Corollary \ref{c1} provide a weak form  of the $H$-theorem for NFPE \eqref{e1.1} in the degenerate case we consider here. Roughly speaking, it amounts to saying that $S(t)u_0\to\oo(u_0)$ in $L^1$ for $t\to\9.$

\begin{remark}
	\label{r1a}\rm One simple example of a function $\beta$ which satisfies all conditions of Theorem \ref{t2.2a} is the following:
	$$\beta(r)=\int^r_0\theta(s)ds,\ \ \ff\,r\ge0,$$where
	$$\theta(s)=\left\{
	\barr{lll}
	-\dd\frac1{\log s}&\mbox{ for }&0<s\le\delta<1,\vsp
	\zeta(s)&\mbox{ for }&s>\delta,\earr\right.$$
	and, for $r\in(-\9,0)$,
	$$\beta(x)=-\beta(-r),$$where $\delta>0$, $\zeta\in C^1[\delta,+\9)$, bounded, $\zeta\ge\zeta_0\in(0,\9)$, and $\zeta$ is such that $\theta\in C^1(0,\9)$.
	Then it is elementary to check that $\beta$  satisfies Hypothesis (i) with $\nu=2$. Furthermore, obviously $g(r)={\rm const}-\log|\log r|$ for $r\in(0,1]$. Hence \eqref{e2.6} also holds and Theorem \ref{t2.2a} applies.
	\end{remark}

\begin{remark}
	\label{r1}\rm Following \cite{1}, we can consider to associate with the dynamics $S(t)$ the Lyapunov function
\begin{equation}\label{e2.12a}
V(u)=\int_\rrd\sigma(u(x))dx+\int_\rrd\Phi(x)u(x)dx=S[u]+E[u],\ \ff\, u\in\calm\cap\calp,\end{equation}where
	$$\sigma(r)=-\int^r_0d\tau\int^1_\tau\frac{\beta'(s)}{sb(s)}\ ds,\ \ \ff\,r\ge0.$$The function $V$ is the free energy of the system,  $S[u]$ is the generalized entropy of the system and $E[u]$ is the internal energy. (In the special case $\beta(r)\equiv r,$ $S[u]$ is just the Boltzmann-Gibbs entropy.) Arguing as in \cite{1}, it follows that $t\to V(S(t)u_0)$ is nonincreasing on $[0,\9)$ and, since $V$ is lower-semicontinuous and positive on $K$, we have by Theorem \ref{t1} that
	$$V(u_\9)=\lim_{t\to\9}V(S(t)u_0)=\inf V,\ \ \ff u_\9\in\oo(u_0).$$\newpage\n Hence, $u_\9$ minimizes the free energy and, as mentioned earlier, this means that $u_\9$ is an equilibrium solution to \eqref{e1.1}. We note that, if $\nabla_x\beta(u(t,\cdot))\in L^2$, then we have 
	$$\frac d{dt}\,S[u(t)]=-\int_{\rr^d}J(u(t,x))\cdot\nabla_x\beta(u(t,x))dx,\ t>0.$$
In particular, this implies that,	  if the current of probability $J$ vanishes at the equilibrium solution $u_\9$, then the   system is in a state of thermodynamic equilibrium with zero entropy production in $u_\9$, that is,  it is reversible \cite{10}.
\end{remark}

In the nondegenerate case \eqref{e1.14},  $u_\9$ is the unique equilibrium state of the system and coincides with the stationary solution to \eqref{e1.1}. It should be said, however, that in our case the equilibrium solution $u_\9$ might not~be~unique.

If $u_0\in L^\9\cap \calm\cap\calp$, then condition \eqref{e1.2a} in Hypothesis (i) can be relaxed~to
$$\mu_1r^\nu\le\beta(r),\ \ \ff\,r>0,\eqno{\eqref{e1.2a}'}$$where $\mu_1>0$ and $\nu>\frac{d-1}d,\ d\ge3.$

\begin{remark}
	\label{r2.6}\rm If, in addition to (i)--(iii), one assumes that $\beta'(r)>\gamma>0$, $\ff r>0$, then  we have that $C=L^1$, hence $\calm_\eta\cap\calp\cap C=\calm_\eta\cap\calp$ and so Theorem \ref{t1} and Corollary \ref{c1} are true for all $u_0\in\calm_\eta\cap\calp.$
\end{remark}

\begin{remark}\label{r2.5} \rm Theorems \ref{t1}, \ref{t2.2a} can be rephrased in terms of the McKean-Vlasov equation \eqref{e1.4a}, that is, in terms of convergence of the time marginal laws of the corres\-ponding probabilistically weak solution $X=X(t)$ for~\mbox{$t\to\9$.}
\end{remark}

\section{Proof of Theorem \ref{t1}}\label{s3}
	\setcounter{equation}{0}

We shall prove Theorem \ref{t1} in three steps indicated by lemmas which follow. 	
	Let $\eta>0$   and let $K=\calm_\eta\cap\calp\cap C$. Clearly, $K$ is a closed and bounded set of $L^1$.
	
\begin{lemma}\label{l3.1}We have	
\begin{equation}
\label{e3.1}
\|(I+\lbb A)\1y\|\le\|y\|,\ \ \ff \,y\in\calm\cap\calp,\ \lbb>0,\end{equation}
and
\begin{equation}
\label{e3.2}
(I+\lbb A)\1(K)\subset\calm_\eta\cap\calp\cap D(A)\subset D(A)\cap K,\ \ff\lbb>0.\end{equation}In particular,
\begin{equation}
	\label{e4.2'}
	\|S(t)u_0\|\le u_0,\ \ \ff u_0\in C,\ t\ge0,\end{equation}and  $S(t)(K)\subset K,\ \ \ff\,t\ge0.$ 
\end{lemma}

\n{\it Proof.} By \cite[Lemma 6.2]{2}, we have $\|J^\vp_\lbb y\|\le\|y\|$ for all $y\in\calm\cap\calp,$ $\lbb>0.$ Hence \eqref{e3.1} follows by Lemma \ref{l22}(iii) and Fatou's lemma. The last assertion then follows by \eqref{e1.12} and \eqref{e1.14a}.\hfill$\Box$ \bk

 In the following, we shall denote by $\wt A$ the restriction of   the operator $A$ to $K$, that is,
\begin{equation}
\label{e3.6a}
\wt Au=Au,\ \ \ff\,u\in D(\wt A)=D(A)\cap K.
\end{equation}
By \eqref{e3.2} it follows that,  for every $\lbb>0$,
\begin{equation}
\label{e4.5}
\ov{D(\wt A)}\subset K\subset (I+\lbb\wt A)(D(\wt A))=R(I+\lbb\wt A)
\end{equation}and we have by definition that
\begin{equation}
\label{e4.6}
(I+\lbb\wt A)\1=(I+\lbb A)\1\mbox{\ \ on }R(I+\lbb\wt A).
\end{equation}Furthermore, $(\wt A,D(\wt A))$ is accretive on $L^1$ and, since $\ov{D(\wt A)}\subset K\subset \ov{D(A)}$, we conclude by \eqref{e1.14a} that $\ff u_0\in\ov{D(\wt A)}$, $t\ge0$,
\begin{equation}
\label{e4.7}
S(t)u_0=\lim_{n\to\9}\(I+\frac tn\ \wt A\)\1u_0\in\ov{D(\wt A)}. 
\end{equation}Therefore, $\wt S(t):=S(t)_{\uparrow\ov{D(\wt A)}}, t\ge0$, is the contraction semigroup generated by $\wt A$ on $\ov{D(\wt A)}$. We are going to apply Theorem 3 in \cite{3} to this semigroup $\wt S(t),$ $t\ge0$, to prove Theorem \ref{t1}. For this we need:

\begin{lemma}\label{l1} The operator $(I+\lbb \wt A)\1$ restricted to $K$ is compact for $\lbb\in(0,\lbb_0]$ with $\lbb_0$ as in Lemma {\rm\ref{l22}}.\end{lemma}

\n{\it Proof.} Let $f_n\in K,\ n\in\nn,$ such that
\begin{equation}\label{e4.8}
\sup_{n\in\nn}|f_n|_1<\9.\end{equation}
Then, since 
$\sup\limits_{n\in\nn}\|f_n\|\le\eta$, by \eqref{e3.1}, \eqref{e4.6} and Lemma A.1, it suffices to~prove that (selecting a subsequence if necessary) for $\lbb\in\!(0,\lbb_0],
J_\lbb f_n=(I+\lbb A)\1f_n$ converges in $L^1_{\rm loc}(K)$ as $n\to\9$ for every compact set $K\subset\rr^d$.  Let  $q\in\(1,\frac d{d-1}\)$. By \eqref{e4.8}, \eqref{e2.15}, \eqref{e2.16} and the Riesz-Kolmogorov compactness theorem, it follows that (selecting a subsequence if necessary) 
  \begin{equation}\label{e4.9}
 \xym{\beta(J_\lbb f_n)\ar[r]_{\ \ \ \ n\to\9}&\eta}\mbox{ in }L^q(K)
 \end{equation}and (since this holds for every such $K$)
 $$\xym{\beta(J_\lbb f_n)\ar[r]_{\ \ \ \ n\to\9}&\eta},\mbox{ a.e. on }\rr^d,$$hence, since $\beta\in C^1$ and $\beta'>0$,
 \begin{equation}
 \label{e4.10}
\xym{J_\lbb f_n\ar[r]_{\!\!\!n\to\9}&\beta\1(\eta)}\mbox{ a.e. on }\rr^d.
\end{equation}Furthermore, because $\nu>\frac{d-1}d$, we may choose $q$ so close to $\frac d{d-1}$ that $\nu q>1$, and hence by~\eqref{e1.2a}
$$\mu^q_1\min(|r|^q,|r|^{\nu q})\le|\beta(r)|^q,\ \ff r\in\rr.$$This implies due to \eqref{e4.9} that $\{J_\lbb f_n\mid n\in\nn\}$ is equi-integrable in $L^1(K)$, hence by \eqref{e4.10} 
$$J_\lbb \xym{f_n\ar[r]_{\!\!\!\!\!n\to\9}&\beta\1(\eta)}\mbox{\ \ in }L^1(K).\eqno\Box$$

\begin{lemma}\label{l3.3} For each $u_0\in K$, the orbit $\gamma(u_0)=\{\wt S(t)u_0,\ t\ge0\}$ is precompact in $L^1$.\end{lemma}
\begin{proof} By Theorem 3 in \cite{3}, applied to the operator $A$ with domain $K$, it suffices to show that
	$\ov{D(\wt A)}\subset R(I+\lbb\wt A),$ the operator $(I+\lbb\wt A)\1$ is compact on $K$ for some  $\lbb>0$ and that the orbit $\gamma(u_0)$ is bounded in $L^1$.
	 In fact, in \cite{3} one assumes that $0\in R(\wt A)$ to have the latter, but  in our case this follows, because $|\wt 
	S(t)u_0|_1=|S(t)u_0|_1\le|u_0|_1,$ since $S(t)$ is a Lipschitz contraction on $L^1$ with $S(t)(0)=0$ (by \eqref{e1.9a} and \eqref{e1.14a}).	
	The first condition in \cite[Theorem 3]{3} is just \eqref{e4.5},    the second is just Lemma \ref{l1}.\end{proof}

\n{\it Proof of Theorem {\rm\ref{t1}}}  ({\it continued}).  Since, by Lemma \ref{l3.3}, $\gamma(u_0)$ is precompact, it follows that $\oo(u_0)\ne\emptyset$    and that $\oo(u_0)$ is compact. Then, the conclusions of the theorem follow by Theorem 1 in \cite{3}.\hfill$\Box$

\section{Appendix}\setcounter{equation}{0}
\n{\bf Proof of Lemma \ref{l22}.}\mk

\n(i): We have seen above that $J^\vp_\lbb f$ is such a solution of \eqref{e2.14}. So, let $u_\lbb,\wt u_\lbb\in L^1\cap L^\9$ be two solutions of \eqref{e2.14}. Then, obviously, $u_\lbb,\wt u_\lbb\in H^1$ and $\beta_\vp(u_\lbb),\beta_\vp(\wt u_\lbb)\in H^2$. Let $u:=u_\lbb-\wt u_\lbb.$ Then, by~\eqref{e2.14},
$$u-\lbb\Delta(\beta_\vp(u_\lbb)-\beta_\vp(\wt u_\lbb))=-\lbb\ {\rm div}(D(b(u_\lbb)u_\lbb-b(\wt u_\lbb)\wt u_\lbb)).$$Applying $\<u,\cdot\>_{-1}$ to both sides of this equation, where $\<\cdot,\cdot\>_{-1}$ denotes an inner product in $H\1$, we find, since $\beta_\vp=\beta+\vp I,$
$$\barr{lcl}
|u|^2_{-1}+\vp\lbb|u|^2_2
&\le&\lbb\<\beta_\vp(u_\lbb)-\beta_\vp(\wt u_\lbb),u\>_{-1}\vsp
&&
-\lbb\<{\rm div}(D(b(u_\lbb)u_\lbb-b(\wt u_\lbb)\wt u_\lbb)),u\>_{-1}\vsp 
&\le&\lbb(\beta_M+C|D|_\9b_M)|u|_2|u|_{-1},\earr\eqno(5.0)$$where we used that $|{\rm div}\cdot|_{-1}\le C|\cdot|_2$ for some $C\in(0,\9)$ and where
	$$\barr{l}
	\beta_M:=\sup\left\{\dd\frac{\beta(r_1)-\beta(r_2)}{r_1-r_2};\ r_1,r_2\le M,\ r_1\ne r_2\right\}+1,\vsp 
	b_M:=\sup\left\{\dd\frac{b(r_1)r_1-b(r_2)r_2}{r_1-r_2};\ r_1,r_2\le M,\ r_1\ne r_2\right\},\earr$$and $M:=\(1+\left|({\rm div}\ D)^-+|D|\right|^{1/2}_\9|f|\).$ We recall that by the proof of Lemma 3.1 in \cite{2} (see formula (3.36)) there exists $\wt\lbb_0\in(0,\9)$ such that for any solution $u_\lbb$ of \eqref{e2.14} we have $|u_\lbb|_\9\le M$, $\ff\lbb\in(0,\wt\lbb_0)$. Hence, 	$\beta_M,b_M<\9$ if $\lbb\in(0,\wt\lbb_0)$. Hence, by Young's inequality there exists $\lbb_0\in(0,\wt\lbb_0)$, such that for some $C_\vp\in(0,\9)$ 
	$$|u|^2_{-1}\le\lbb C_\vp|u|^2_{-1},\ \ff\lbb\in(0,\lbb_0]$$Hence, $u=0.$ 

\n	(ii):  We know by \eqref{2.8'} and \eqref{e29} that  for $\lbb\in(0,\9)$ 
	$$a_\vp+\lbb(A_0)_\vp a_\vp=a_\vp$$and that by (i) for $\lbb\in(0,\lbb_0]$
	$$J^\vp_\lbb a_\vp+\lbb(A_0)_\vp J^\vp_\lbb a_\vp=a_\vp.$$Hence by the uniqueness part of (i) assertion (ii) follows since $a_\vp\in L^1\cap L^\9$ by \eqref{e28}, \eqref{e29}.\bk
	
	\n To prove (iii), for each  $f\in L^1$, we set $u_\vp=(I+\lbb A_\vp)\1f$, that is
\begin{equation}
\label{e4.1}u_\vp-\lbb\Delta\beta_\vp(u_\vp)+\lbb\,{\rm div}(Db(u_\vp)u_\vp)=f\mbox{\ \ in }\cald'(\rr^d).
\end{equation}To prove that, for $\vp\to0$ it follows that $u_\vp\to u$ in $L^1$, where $u$ is a  solution~to
\begin{equation}\label{e4.2}
u-\lbb\Delta\beta(u)+\lbb\,{\rm div}(Db(u)u)=f\mbox{\ \ in }\cald'(\rr^d),
\end{equation}where $\beta_\vp(u)=\beta(u)+\vp u$  (see \eqref{e23'}),  we shall proceed as in the proof of Lemmas  3.1 and 3.3 in \cite{2}. However, since the proof is the same, it will be sketched only. Namely, using the fact that $\beta'>0$, it follows that $u_\vp\to u\in(I+\lbb A_0)\1f$ strongly in $L^1_{\rm loc}$ and $\beta(u_\vp)\to\beta(u)$ in $L^1_{\rm loc}$.

 Now, let $\wt\Phi\in C^2(\rr^d)$ be such that $\wt\Phi(x)\to\9$ as $|x|\to\9$, $\nabla\wt\Phi\in L^\9,$ $\Delta\wt\Phi\in L^\9$.

 If we multiply \eqref{e4.1} by $\wt\Phi\exp(-\nu\wt\Phi)\calx_\delta(\wt\beta_\vp(u_\vp))$, where $\nu>0$, and integrate on $\rr^d$, we get  after some calculation identical with that in the proof of Lemma 3.3 in \cite{2} that, for each $f\in\wt\calm$, we have the estimate
	$$\|u_\vp\|_*\le\|f\|_*+C\lbb(|\Delta\wt\Phi)|_\9+|D|_\9|\nabla\wt\Phi|_\9)|f|_1,$$where
	$$\|u\|_*=\int_{\rr^d}|u(x)|\wt\Phi(x)dx,\ \wt\calm=\{f\in L^1;\,\|f\|_*<\9\}.$$By Lemma A.1 below, this implies that, for $\vp\to0$, $u_\vp\to u$ in $L^1$ and, letting $\vp\to0$ in \eqref{e4.1}, we get that $u$ is a solution to \eqref{e4.2}, as claimed.\hfill$\Box$\bk
	
	\n(iv): It follows from Lemma 3.1 in \cite{2} and its proof that there exists \mbox{$C\!\in\!(0,\9)$} only depending on $K,q,\lbb,|D|_\9$ and $|b|_\9$ such that for all $f\in L^1\cap L^\9$, $\nu\in\rr^d$, $\vp\in(0,1],$ $|J^\vp_\lbb f|_\9\le C|f|_\9,$
	\begin{equation}\label{e6.3}
	\|\beta_\vp(J^\vp_\lbb f)\|_{L^q(K)}\le C|f|_1\end{equation}and
		\begin{equation}\label{e6.4}
	\|\beta_\vp(J^\vp_\lbb f)^\nu\|_{L^q(K)}\le C\(\|E^\nu\|_{M^{\frac d{d-2}}}+\|\nabla E^\nu\|_{M^{\frac d{d-1}}}\)|f|_1.\end{equation}
		We have that $\lim\limits_{n\to\9}\|E^\nu\|_{M^{\frac d{d-2}}}+\|\nabla E^\nu\|_{M^{\frac d{d-1}}}=0$, hence, by (iii) and the Riesz-Kolmogorov compactness theorem, along a subsequence $\vp\to0$,
		$$\barr{rcll}
		J^\vp_\lbb f&\to&J_\lbb f&\mbox{ weakly in }L^q(K),\vsp 
		\beta(J^\vp_\lbb f)&\to&\eta&\mbox{ strongly in }L^q(K).\earr$$
		Since $u\mapsto\beta(u)$ is maximal monotone in each dual pair $(L^q(K), L^{q'}(K))$, hence weakly-strongly closed, we conclude that
		$$\eta=\beta(J_\lbb f).$$Hence, by Fatou's lemma we may pass to the limit in \eqref{e6.3}, \eqref{e6.4} to obtain \eqref{e2.15} and \eqref{e2.16} for all $f\in L^1\cap L^\9$. Since $L^1\cap L:^\9$ is dense in $L^1$, applying Fatou's lemma again we obtain \eqref{e2.15}, \eqref{e2.16} for all $f\in L^1$. $\Box$

\bk We use the following well known result in this paper. We include a proof for the reader's convenience.

\bk\n{\bf Lemma A.1.} {\it Let $u_n\in L^1(\rr^d)$, $n\in\nn$, such that for some $u\in L^1_{\rm loc}(\rr^d)$
	$$\lim_{n\to\9}u_n=u\mbox{\ \ in }L^1_{\rm loc}(\rr^d).$$Furthermore, let $\Phi:\rrd\to[1,\9)$ be Borel-measurable with $\{\Phi\le c\}$ relatively compact for all $c\in(0,\9)$ such that
	$$\sup_{n\in\nn}\int_\rrd|u_n|\Phi\,dx<\9.\eqno{\rm(A.1)}$$Then $\int_\rrd|u|\Phi\,dx<\9$ and $\lim\limits_{n\to\9}u_n=u$ in $L^1(\rrd).$}

\begin{proof} By Fatou's Lemma and (A.1)
	$$\int_\rrd|u|\Phi\,dx\le\sup_{n\in\nn}\int_\rrd|u_n|\Phi\,dx<\9.$$Hence, $u\in L^1(\rrd)$ and, for all $c\in(0,\9)$,
	$$\barr{l}
	\dd\limsup_{n\to\9}\int_\rrd|u{-}u_n|dx
	=\dd\limsup_{n\to\9}\int_{\{\Phi>c\}}|u{-}u_n|dx
	+\dd\limsup_{n\to\9}\int_{\ov{\{\Phi\le c\}}}|u{-}u_n|dx\vsp
	\ \hspace*{40mm}\le\dd\frac1c\sup_n\int_\rrd(|u|+|u_n|)\Phi\,dx\xym{{}\ar[r]_{c\to\9}&0}.\earr$$
\end{proof}

\bk\n{\bf  Lemma A.2.} {\it Assume that Hypotheses {\rm(i), (ii), (iii)} hold. Let $g$ be as in \eqref{e2.4a}, i.e.,
$$g(r)=\int^r_1\frac{\beta'(s)}{sb(s)}\,ds,\ \ r\in(0,\9).$$
\begin{itemize}
	\item[\rm(j)] Let $\nu\in\(1-\frac1d,1\right].$ Then
$$\lim_{r\to0}g(r)=-\9.\eqno{\rm(A.2)}$$
If, in addition, $\lim\limits_{r\to\9} g(r)=\9$, then
$$g\1(r)\le e^{\frac1{\mu_2}(b_0r+\beta(1)-\mu_1)},\ \ff r\in\(-\9,\mbox{$\frac1{b_0}$}(\mu_1-\beta(1))\right],\eqno{\rm(A.3)}$$hence, for all $\mu\in\rr$ on all of $\rrd$,
$$\barr{ll}
g\1(\mu-\Phi)\!\!\!&
\le 1_{\{\Phi\ge\mu+b^{-1}_0(\beta(1)-\mu_1)\}}
e^{\frac1{\mu_2}(b_0\mu+\beta(1)-\mu_1)}
e^{-\frac{b_0}{\mu_2}\,\Phi}\vsp
&+\,1_{\{\Phi<\mu+b^{-1}_0(\beta(1)-\mu_1)\}}
g\1(\mu-\Phi),\earr\eqno{\rm(A.4)}$$
which in turn implies that, for all $\mu\in\rr,$
$$\barr{r}
\dd\int_\rrd\!\! g\1(\mu{-}\Phi(x))\Phi(x)dx
 \le e^{\frac1{\mu_2}(b_0\mu+\beta(1)-\mu_1)}
\dd\int_\rrd
e^{-\frac{b_0}{\mu_2}\,\Phi(x)}\Phi(x)dx\vsp
 +\,g\1(\mu-1)\dd\int_{\{\Phi<\mu+b^{-1}_0(\beta(1)-\mu_1)\}}\Phi(x)dx<\9.\earr\eqno{\rm(A.5)}$$

\item[\rm(jj)] Let $\nu\in(1,\9).$ Then,
$$\lim_{r\to\9} g(r)=\9.\eqno{\rm(A.6)}$$If, in addition, $\lim\limits_{r\to0}g(r)=-\9,$ then	
$$g\1(r)\le e^{\frac1{\mu_2}(b_0r+\beta(1))},\ \ff r\in\(-\9,-\mbox{$\frac{\beta(1)}{b_0}$}\right],\eqno{\rm(A.7)}$$hence, for all $\mu\in\rr$ on all of $\rrd$,
$$\barr{ll}
g\1(\mu-\Phi)\!\!\!&
\le 1_{\left\{\Phi\ge\mu+\frac{\beta(1)}{b_0}\right\}}
e^{\frac1{\mu_2}(b_0\mu+\beta(1))}
e^{-\frac{b_0}{\mu_2}\Phi}\vsp
&+\, 1_{\left\{\Phi<\mu+\frac{\beta(1)}{b_0}\right\}}g\1(\mu-\Phi),\earr\eqno{\rm(A.8)}$$which in turn implies  that, for all $\mu\in\rr,$
$$\barr{ll}
\hspace*{-6mm}\dd\int_\rrd\!\! g\1(\mu{-}\Phi(x))\Phi(x)dx\!\!\!
&\le  e^{\frac1{\mu_2}(b_0\mu+\beta(1))}\dd\int_\rrd\!\!
e^{-\frac{b_0}{\mu_2}\Phi(x)}\Phi(x)dx\vsp&
+g\1(\mu-1)\dd\int_{\left\{\Phi<\mu+\frac{\beta(1)}{b_0}\right\}}\Phi(x) dx<\9.
\earr\eqno{\rm(A.9)}$$
\end{itemize}}

\begin{proof} First, we note that by Hypothesis (iii) and integrating by parts we find
$$1_{(0,1]}\,\frac{\wt g}{b_0}+
1_{(1,\9)}\,\frac{\wt g}{|b|_\9}
\le g\le 1_{(0,1]}\,\frac{\wt g}{|b|_\9}
+1_{(1,\9)}\,\frac{\wt g}{b_0},
\eqno{\rm(A.10)}$$where
$$\wt g(r):=\frac{\beta(r)}r-\beta(1)+\int^r_1\frac{\beta(s)}{s^2}\,ds,\ r\in(0,\9).$$

\n(j): Let $\nu\in\(1-\frac1d,1\right].$ Then, by  \eqref{e1.2a}, for $r\in(0,1]$,
		$$\wt g(r)\le\mu_2-\beta(1)-\int^1_r\frac{\mu_1}s\,ds.$$
So, (A.10) implies (A.2). Now, additionally assume that $\lim\limits_{r\to\9}g(r)=\9$, so~that $g:(0,\9)\to\rr$ is bijective. Again by (A.10) and \eqref{e1.2a} we obtain, for all $r\in(0,\9),$
$$g(r)\ge b^{-1}_0 1_{(0,1]}(r)(\mu_1-\beta(1)+\mu_2\ln r)+1_{(1,\9)}(r)g(r).$$
Replacing $r\in(0,\9)$ by $e^{\frac1{\mu_2}(b_0r+\beta(1)-\mu_1)}\in(0,1]$,  for $r\in(-\9,b^{-1}_0(\mu_1-\beta(1))],$ we obtain, since $g$  is increasing,
$$g\(e^{\frac1{\mu_2}(b_0r+\beta(1)-\mu_1)}\)\ge r,$$and thus
$$g\1(r)\le e^{\frac1{\mu_2}(b_0r+\beta(1)-\mu_1)},\ \ff r\in(-\9,b^{-1}_0(\mu_1-\beta(1))],$$which, for all $\mu\in\rr,$ implies that on all of $\rrd$
$$\barr{ll}
g\1(\mu-\Phi)\!\!\!&
\le1_{\{\Phi\ge\mu+b^{-1}_0(\beta(1)-\mu_1)\}}
e^{\frac1{\mu_2}(b_0\mu+\beta(1)-\mu_1)}
e^{-\frac{b_0}{\mu_2}\,\Phi}\vsp
&+\,1_{\{\Phi<\mu+b^{-1}_0(\beta(1)-\mu_1)\}}g\1(\mu-\Phi),\earr$$which is (A.4). (A.5) is now obvious by \eqref{e1.2} (which, in particular, implies that $\{\Phi<c\}$ is relatively compact $\ff c\in\rr$) and because $g\1$ is increasing.

\bk\n(jj): Let $\nu\in(1,\9).$ Then, by \eqref{e1.2a}, for $r\in(1,\9)$,
$$\wt g(r)\ge\mu_1-\beta(1)+\mu_1
\int^r_1\frac{1}s\,ds.$$
So, (A.10) implies (A.6). Now, additionally assume that $\lim\limits_{r\to0}g(r)=-\9$, so~that $g:(0,\9)\to\rr$ is bijective. Again by (A.10) and \eqref{e1.2a} we obtain, for all $r\in(0,\9),$
$$\barr{ll}
g(r)\!\!\!&\ge 1_{(0,1]}(r)b^{-1}_0
\(\mu_1r^{\nu-1}-\beta(1)-\dd\int^1_r\frac{\mu_2}s\,ds\)+1_{(1,\9)}(r)g(r)\vsp
&\ge1_{(0,1]}b^{-1}_0(\mu_2\ln r-\beta(1))+1_{(1,\9)}(r)g(r).\earr$$
Replacing   $r\in(0,\9)$ by $e^{\frac1{\mu_2}(b_0r+\beta(1))}\in(0,1]$ for    $r\in\(-\9,-\frac{\beta(1)}{b_0}\right],$  we obtain
$$g\(e^{\frac1{\mu_0}(b_0r+\beta(1))}\)\ge r,$$and thus
$$g\1(r)\le e^{\frac1{\mu_2}(b_0r+\beta(1))},\ \ff r\in\(-\9,-\frac{\beta(1)}{b_0}\right],$$  which, for all $\mu\in\rr,$ implies that on all of $\rrd$
$$g\1(\mu-\Phi)
\le1_{\left\{\Phi\ge\mu+\frac{\beta(1)}{b_0}\right\}}
e^{\frac1{\mu_2}(b_0\mu+\beta(1))}
e^{-\frac{b_0}{\mu_2}\,\Phi}
+1_{\left\{\Phi<\mu+\frac{\beta(1)}{b_0}\right\}}g\1(\mu-\Phi).$$Hence, (A.7), (A.8) are proved and   (A.9) is then again an obvious consequence.
\end{proof}

\mk\n{\bf Acknowledgements.} This work was supported by the DFG through CRC 1283.

\end{document}